\documentclass[onefignum,onetabnum]{siamart250211}
\usepackage[utf8]{inputenc}
\usepackage{a4wide}

\usepackage{amsfonts}
\usepackage{amsmath}
\usepackage[mathscr]{eucal}
\usepackage{amssymb}
\usepackage{latexsym}
\usepackage{amscd}
\usepackage{epsf}
\usepackage{graphicx}
\usepackage{makeidx}
\usepackage{enumerate}
\usepackage{mathdots}
\usepackage{color}
\usepackage{hyperref}
\usepackage{url}
\usepackage[labelformat=parens]{subcaption}  

\newsiamremark{remark}{Remark}

\newcommand{\bsb}{\boldsymbol{b}}

\newcommand{\bse}{\boldsymbol{e}}

\newcommand{\bsh}{\boldsymbol{h}}

\newcommand{\bsk}{\boldsymbol{k}}
\newcommand{\bsl}{\boldsymbol{l}}

\newcommand{\bst}{\boldsymbol{t}}

\newcommand{\bsx}{\boldsymbol{x}}
\newcommand{\bsy}{\boldsymbol{y}}
\newcommand{\bsz}{\boldsymbol{z}}

\newcommand{\bszero}{\boldsymbol{0}}
\newcommand{\bsone}{\boldsymbol{1}}

\newcommand{\Scal}{\mathcal{S}}

\newcommand{\NN}{\mathbb{N}}

\newcommand{\RR}{\mathbb{R}}

\newcommand{\Zb}{\mathbb{Z}_b}

\newcommand{\Fb}{\mathbb{F}_b}

\newcommand{\wal}{\mathrm{wal}}
\newcommand{\Var}{\operatorname{Var}}
\newcommand{\BOX}{\operatorname{BOX}}
\newcommand{\bcoarse}{{\mathrm{C}}}
\newcommand{\busual}{{\mathrm{U}}}
\newcommand{\Vcoarse}{\sigma^2_{\mathrm{C}}}
\newcommand{\Vusual}{\sigma^2_{\mathrm{U}}}

\DeclareMathOperator*{\argmin}{arg\,min}

\newenvironment{enuroman}{\begin{enumerate}[\normalfont (i)]}{\end{enumerate}}

\title{Coarse scrambling for Sobol' and Niederreiter sequences\thanks{Submitted to the editors 2025 October 3.
\funding{The work was supported by JSPS KAKENHI Grant Number 24K06857.}}}

\author{Kosuke Suzuki\thanks{Faculty of Science, Yamagata University, 1-4-12 Kojirakawa-machi, Yamagata, 990-8560, Japan  (\email{kosuke-suzuki@sci.kj.yamagata-u.ac.jp}).}
}

\headers{Coarse scrambling}{Kosuke Suzuki}

\begin{document}



\date{\today}
\maketitle

\begin{abstract}
We introduce \emph{coarse scrambling}, a novel randomization for digital sequences that permutes blocks of digits in a mixed-radix representation. This construction is designed to preserve the powerful $(0,\bse,d)$-sequence property of the underlying points. For sufficiently smooth integrands, we prove that this method achieves the canonical $O(n^{-3+\epsilon})$ variance decay rate, matching that of standard Owen's scrambling. Crucially, we show that its maximal gain coefficient grows only logarithmically with dimension, $O(\log d)$, thus providing theoretical robustness against the curse of dimensionality affecting scrambled Sobol' sequences. Numerical experiments validate these findings and illustrate a practical trade-off: while Owen's scrambling is superior for integrands sensitive to low-dimensional projections, coarse scrambling is competitive for functions with low effective truncation dimension.
\end{abstract}


\begin{keyword}
Quasi-Monte Carlo, scrambling, digital net, gain coefficient, Sobol' sequence
\end{keyword}

\begin{MSCcodes}
65C05, 65D30
\end{MSCcodes}

\begin{center}
\textcolor{red}{
This is the final pre-publication version of the manuscript, made available on arXiv as Green Open Access. The final published version may differ.
}\end{center}

\section{Introduction}\label{sec:intro}

We consider numerical approximation of $d$-dimensional integrals of the form
\begin{equation}\label{eq:target_integral}
I(f) := \int_{[0,1)^d} f(\boldsymbol{x})\,d\boldsymbol{x},
\end{equation}
for Riemann-integrable functions $f:[0,1)^d\to\mathbb{R}$.
A common practical estimator is the equally weighted quadrature rule
\begin{equation}\label{eq:qmc-estimator}
I(f;P_n) := \frac{1}{n}\sum_{i=0}^{n-1} f(\boldsymbol{x}_i),
\end{equation}
where $P_n=\{\boldsymbol{x}_0,\dots,\boldsymbol{x}_{n-1}\}\subset[0,1)^d$ is a collection of sample points.
When the $\boldsymbol{x}_i$ are independent uniform samples,
\eqref{eq:qmc-estimator} reduces to the classical Monte Carlo (MC)
estimator which is unbiased and has variance
\begin{equation}\label{eq:MCvar}
\Var(I(f;P_n)) = \frac{\Var{f}}{n},
\qquad
\text{where } \Var f := I(f^2) - I(f)^2,
\end{equation}
for $f\in L^2([0,1)^d)$.
Despite its robustness, the root mean square error of the MC estimator converges at the relatively slow rate of $O(n^{-1/2})$,
motivating the use of more structured point sets.

Quasi-Monte Carlo (QMC) methods replace independent sampling by carefully
constructed low-discrepancy point sets such as lattice rules \cite{SJ94,DKP22book}, Halton sequences \cite{Ha60},
and digital nets \cite{Ni92book,DP10book}(e.g. Sobol' \cite{So67} and Niederreiter sequences \cite{Ni88}).
For integrands with bounded variation, these constructions often achieve much faster error decay
than MC as $O(n^{-1+\epsilon})$.
A key drawback of deterministic QMC, however, is the lack of a practical and reliable error estimate. Randomized QMC (RQMC) resolves this issue by introducing a randomization that preserves the low-discrepancy structure of the points while restoring the benefits of a stochastic method: the estimator becomes unbiased, and its variance can be estimated from independent replications, thus providing practical error bars.

Among randomized QMC schemes, Owen's scrambling \cite{Ow95, Ow97a} has
become a standard method for digital nets.
In Owen's scrambling, each digit of an input point's base-$b$ expansion is permuted
by an independently chosen random permutation (possibly dependent on higher-order digits).
For a deterministic point set $P_n$, applying Owen's scrambling to each point yields a randomly scrambled point set $\widetilde{P}_n$.
Owen proved a bound of the form
\begin{equation}\label{eq:gain-intro2}
\Var\bigl(I(f;\widetilde{P}_n)\bigr)\le \frac{\Gamma(P_n)}{n}\Var(f),
\end{equation}
where $\Gamma(P_n)$ is the \emph{maximal gain coefficient}
associated with the point set.
Moreover, for sufficiently smooth integrands the variance decays substantially faster as $O(n^{-3 + \epsilon})$ (see \cite{Ow97b,Ow98}),
explaining the large practical gains of scrambling over plain MC.
While recent work has also focused on alternative estimators, such as the median-of-means, to achieve robust and almost optimal order of convergence without a knowledge of integrands \cite{PO23, PO24, GSM24, Pan25},
the focus of this paper is on the variance properties of the classical mean-based estimator.

In practice, \emph{scrambled Sobol' sequences} \cite{So67,HH03,JK08} are arguably the most widely used in scientific computing and industry due to their excellent empirical performance.
A theoretical drawback, however, is that their maximal gain coefficient $\Gamma(P_n)$ can grow exponentially with dimension $d$ \cite{PO21, GS22}. In very high dimensions, this can lead to a situation where scrambled QMC performs substantially worse than plain Monte Carlo, undermining its primary motivation.

Some alternative RQMC constructions yield smaller maximal gain coefficients.
For example, scrambled Faure sequences in prime power base \(b\) (with \(d \leq b\)) satisfy \(\Gamma(P_n) \leq e \approx 2.718\) \cite{Ow97a},
but they are generally not considered efficient.
One reason is that, numerically, the benefits of scrambling appear only when \(n \gtrsim b^{d}\).
Other drawbacks are that large bases are required in high dimensions, and arithmetic in base \(b\) is slower than in base \(2\).
Scrambled Halton sequences also have small gain coefficients, growing only as \(O(\log d)\) \cite{OP24gain},
but they typically do not achieve the higher-order variance decay observed for scrambled digital nets.

This paper introduces \emph{coarse scrambling}, a randomization technique for digital sequences designed to operate on blocks of digits rather than individual digits. The motivation for this block-wise approach stems directly from our primary goal: to preserve the powerful $(0,\bse,d)$-sequence property of Sobol' and Niederreiter sequences. Since this property is defined over mixed bases of the form $(b^{e_1},\dots,b^{e_d})$,
scrambling entire blocks of digits is the natural and necessary way to maintain this structure—a feat not achieved by standard digit-wise scrambling. By preserving this crucial property, our construction achieves two critical objectives: it controls the maximal gain coefficient to grow only logarithmically with dimension, $O(\log d)$, while simultaneously retaining the $O(n^{-3+\epsilon})$ convergence rate for smooth functions.

Our main contributions are as follows:
\begin{enumerate}
\item \textbf{Coarse Scrambling with Theoretical Performance Guarantees:}
Building upon Tezuka's framework of sequences in multiple bases~\cite{Tezuka2016trac},
we adopt the notion of $(\bst,d)$-sequences in multiple bases.
Within this setting, we introduce \emph{coarse scrambling} for digital $(0,\boldsymbol{e},d)$-sequences in base $b$ (e.g., Sobol' and Niederreiter), implemented via the mixed base $(b^{e_1},\dots,b^{e_d})$, and prove that it preserves the underlying $(0,\boldsymbol{e},d)$-structure.
By extending recent gain-coefficient formulas for scrambled Halton sequences~\cite{OP24gain}, we derive an explicit expression for the gain coefficients of the coarsely scrambled digital sequences and establish two consequences:
(i) for sufficiently smooth integrands, the variance decays at the canonical rate $O(n^{-3+\epsilon})$, and
(ii) the maximal gain coefficient grows only logarithmically with the dimension, $O(\log d)$.

    \item \textbf{Empirical Validation of the Theoretical Trade-offs:} We conduct numerical experiments to validate our theoretical findings and illustrate the practical trade-off between coarse and usual scrambling. The results highlight that the superior accuracy of usual scrambling is driven by its excellent performance on low-dimensional projections. Conversely, for functions with low effective dimensionality, where the total variance is concentrated in the first few coordinates, the performance of coarse scrambling is comparable. We also present an experiment based on a 19-dimensional function unfavorable to usual scrambling for which coarse scrambling is markedly more efficient than usual scrambling, illustrating the practical impact of its much slower growth of maximal gain coefficients.
\end{enumerate}

The remainder of the paper is organized as follows.
Section~\ref{sec:prelim} recalls basic definitions and notation for digital nets, scrambling, and gain coefficients.
In Section~\ref{sec:gain} we develop the theoretical analysis of
gain coefficients for coarse scrambling.
Section~\ref{sec:comparison} compares the usual (digitwise) scrambling and coarse scrambling theoretically.
Section~\ref{sec:experiments} reports numerical experiments.

\section{Preliminaries}\label{sec:prelim}

We use the following notation.
Let $\NN$ be the set of positive integers, $\NN_0 := \NN \cup \{0\}$,
and $\Zb := \{0,1,\dots, b-1\}$.
Let $\chi(A)$ denote the characteristic function of the event $A$.
Let $1{:}d := \{1, \dots, d\}$.
For \(u \subseteq 1{:}d\), we denote the cardinality of \(u\) by \(|u|\).
Let $\bszero := (0,\dots,0)$ denote the zero vector; its dimension will be clear from the context.

\subsection{Digital Constructions and Equidistribution}\label{sec:construction}

Digital constructions in prime power base $b$ were introduced by Niederreiter.
We adopt the definition given in \cite{DP10book}.
Let $\Fb$ be a $b$-element field and
fix a bijection $\phi \colon \Zb \to \Fb$ with $\phi(0) = 0$.
If $b$ is a prime, we usually identify $\Fb$ with $\Zb$ with addition, subtraction, and multiplication defined modulo $b$.

\begin{definition}
Let $b$ be a prime power and $C_1, \dots, C_d \in \Fb^{\infty \times \infty}$.
For an integer $k \ge 0$,
denote its $b$-adic expansion by $k = \sum_{i=0}^{\infty} \kappa_i b^{i}$ with $\kappa_i \in \Zb$
and let $\vec{k} = (\phi(\kappa_0), \phi(\kappa_1), \dots)^\top$.
Then the $k$-th point $\bsx_k\in [0,1)^d$ is given by
\begin{align}\label{eq:pt-generation}
\bsx_k := (\psi(C_1\vec{k}), \dots, \psi(C_d\vec{k})),
\end{align}
where the map $\psi \colon \Fb^\infty \to [0,1]$ is defined by
\begin{equation}\label{eq:vec2val}
\psi((y_1,y_2, \dots)^\top) := \sum_{i=1}^\infty \frac{\phi^{-1}(y_{i})}{b^i}.
\end{equation}
The sequence of points $\{\bsx_0,\bsx_1 \dots\}$ constructed this way is called a \textit{digital sequence} in base $b$ with generating matrices $C_1,\dots,C_d$.
\end{definition}

To assess the quality of point sets and sequences, the notions of $(t,m,d)$-nets and $(t,d)$-sequences are widely used to establish their equidistribution property.
To introduce these objects, we first define elementary intervals.

\begin{definition}
Let $b_j, k_j \in \NN_0$ with $b_j \ge 2$ for $1 \le j \le d$.
The \textit{elementary $(k_1, \dots, k_d)$-interval in base $(b_1, \dots, b_d)$}
is an interval of the form
\[
\prod_{j=1}^d \left[\frac{a_j}{b_j^{k_j}}, \frac{a_j+1}{b_j^{k_j}}\right) \quad \text{with $a_j \in \NN_0$,\, $0 \le a_j < b_j^{k_j}$}.
\]
\end{definition}

\begin{definition}
Let $m,b,t \in \NN_0$ with $m \ge 1$ and $b \ge 2$.
Let $P = \{\bsx_0, \bsx_1, \dots, \bsx_{b^m-1}\} \subset [0,1)^d$ be a set of $b^m$ points.
We call $P$ a \textit{$(t,m,d)$-net in base $b$} if
every elementary $(k_1, \dots, k_d)$-interval in base $(b,\dots,b)$ with $k_1 + \cdots + k_d = m-t$
contains exactly $b^t$ points of $P$.
A sequence $\Scal = \{\bsx_0, \bsx_1, \dots\} \subset [0,1)^d$ is called a \textit{$(t,d)$-sequence in base $b$}
if, for any $m,r \in \NN_0$ with $m \ge t$,
the point set $\{\bsx_{rb^m}, \dots, \bsx_{(r+1)b^m-1}\}$ is a $(t,m,d)$-net.
\end{definition}

\subsection{Generalized Niederreiter Sequences}

A well-known and important class of digital sequences is given by the Sobol' and (generalized) Niederreiter sequences.
The Sobol' sequence, defined in \cite{So67}, can be regarded as a generalized Niederreiter sequence in base \(2\)
whose base polynomials are chosen to be primitive (see \cite{FL16} for details).
The Niederreiter sequence was introduced in \cite{Ni88}
and later generalized by Tezuka in \cite{Te93},
by allowing more general polynomials $y^{(j)}_k(x)$ (see Definition~2.4).

\begin{definition}
Let $b$ be a prime power and let the base polynomial $p_1(x),p_2(x),\dots,p_d(x)$
be monic polynomials in $\mathbb{F}_b[x]$ which are pair-wise coprime. Let $e_j := \deg(p_j)$.
For $k \in \mathbb{N}$ and $1 \le j \le d$, let $k_j=\lfloor (k-1)/e_j \rfloor$ and let
$y^{(j)}_{k}(x)\in \mathbb{F}_b[x]$ with the restriction that for each $q\in\mathbb{N}$,
the residue polynomials
$\{
y^{(j)}_{k}(x)\bmod p_j(x) \mid (q-1)e_j \le k-1 < q e_j
\}$
are linearly independent over $\mathbb{F}_b$.
Consider the Laurent series expansion
\[
\frac{y^{(j)}_{k}(x)}{(p_j(x))^{k_j+1}}
=
\sum_{r=1}^{\infty} a^{(j)}(k_j+1,k,r)\,x^{-r}\in \mathbb{F}_b((x^{-1})).
\]
The \textit{generalized Niederreiter sequence} is
defined as a digital sequence whose generating matrices are
$C_j=(c_{k,r}^{(j)})_{k,r \in \NN}$ for $1 \le j \le d$,
where each element is given by
\begin{equation*}\label{eq:general-Nied-seq}
c_{k,r}^{(j)} = a^{(j)}\left(k_j+1, k, r\right).
\end{equation*}
\end{definition}

\begin{definition}
We use the term \textit{full Niederreiter sequence} to refer to the generalized Niederreiter sequence in which the base polynomials $p_1(x), p_2(x), \dots$ are taken to be all the monic irreducible polynomials over $\Fb$, ordered by increasing degree.
\end{definition}

Generalized Niederreiter sequences are important examples of $(t,d)$-sequences, as stated in the following theorem \cite[Section~8.1]{DP10book}.
\begin{theorem}\label{thm:t-val-nied}
Any generalized Niederreiter sequence in prime power base $b$ with the base polynomials $p_1(x),\dots,p_d(x)$
is a $(t,d)$-sequence with
$t = \sum_{j=1}^d (\deg(p_j) - 1)$.
\end{theorem}

As the theorem indicates, the quality of the generalized Niederreiter sequence depends on the degrees of the base polynomials. To achieve a small $t$-value, one must choose base polynomials with low degrees.
Table~\ref{table:degree} summarizes the number of available irreducible and primitive polynomials over $\mathbb{F}_2$ for a given degree.
For example, when constructing a Sobol' sequence (where primitive polynomials are typically used), if one wishes to keep the degrees of the base polynomials at most $7$, the maximum possible dimension is $d = 1 + \sum_{n=1}^{7} P_2(n) = 37$, where $P_2(n)$ denotes the number of primitive polynomials over $\mathbb{F}_2$ of degree $n$, and the first dimension is chosen separately with $p_1(x)=x$.

\begin{table}[ht]
\centering
\begin{tabular}{c|c*{14}{r}}
$n$ & 1 & 2 & 3 & 4 & 5 & 6 & 7 & 8 & 9 & 10 & 11 & 12 & 13 & 14 \\
\hline
irreducible &
  2 & 1 & 2 & 3 & 6 & 9 & 18 & 30 & 56 & 99 & 186 & 335 & 630 & 1161 \\
primitive &
  1 & 1 & 2 & 2 & 6 & 6 & 18 & 16 & 48 & 60 & 176 & 144 & 630 & 756 \\
\end{tabular}
\caption{Number of monic irreducible and primitive polynomials over $\mathbb{F}_2$.}
\label{table:degree}
\end{table}

\subsection{Equidistribution and Further Generalization}

A smaller $t$-value is preferable, since it indicates that the point set satisfies an equidistribution property with respect to finer partitions of $[0,1)^d$.
However, the ideal case $t=0$ is often unattainable in high dimensions.
For instance, if $b$ is prime, a $(0,b+1)$-sequence in base $b$ does not exist.
To overcome this limitation, generalized notions of
$(t,m,\bse,d)$-nets and
$(t,\bse,d)$-sequences were introduced \cite{tezuka2013discrepancy}. In these notions,
the class of elementary intervals under consideration is restricted,
which in turn allows for smaller $t$-values.

\begin{definition}
Let $m,b,t \in \NN_0$ with $m \ge 1$ and $b \ge 2$ and $e_j \in \NN$ for $1 \le j \le d$.
Let $P = \{\bsx_0, \bsx_1, \dots, \bsx_{b^m-1}\} \subset [0,1)^d$ be a set of $b^m$ points.
We call $P$ a \textit{$(t,m,\bse,d)$-net in base $b$} if
every elementary $(k_1, \dots, k_d)$-interval in base $(b^{e_1},\dots,b^{e_d})$ with $e_1k_1 + \cdots + e_dk_d = m-t$
contains exactly $b^t$ points of $P$.
A sequence $\Scal = \{\bsx_0, \bsx_1, \dots\} \subset [0,1)^d$ is called a \textit{$(t,\bse,d)$-sequence in base $b$}
if, for any $m,r \in \NN_0$ with $m \ge t$,
the point set $\{\bsx_{rb^m}, \dots, \bsx_{(r+1)b^m-1}\}$ is a $(t,m,\bse,d)$-net.
\end{definition}

Under this generalized framework, generalized Niederreiter sequences exhibit an ideal equidistribution property, as shown by Tezuka \cite[Theorem~1]{tezuka2013discrepancy}.
\begin{theorem}\label{thm:Nied_is_0ed}
Let $\Scal$ be a generalized Niederreiter sequence in prime power base $b$ with the base polynomials $p_1(x),\dots,p_d(x)$.
Let $\bse = (\deg(p_1), \dots, \deg(p_d))$.
Then $\Scal$ is a $(0,\bse,d)$-sequence in base $b$.
\end{theorem}

Tezuka extended the notion of $(t,\bse,d)$-sequences in a single base $b$ to $(\bst,\bse,d)$-sequences in multiple bases $\bsb=(b_1,\dots,b_d)$~\cite{Tezuka2016trac}.
In this paper, we restrict attention to the special case $\bse=\mathbf{1}:=(1,\dots,1)$.
Accordingly, we introduce the shorthand ``$(\bst,d)$-sequence in multiple bases $\bsb$'' for a $(\bst,\mathbf{1},d)$-sequence in base $\bsb$.
Note that $(t,d)$-sequence in base $b$
is a $((t_1,\dots,t_d),d)$-sequence in base $(b,\dots,b)$ if $t_1 + \dots + t_d = t$.
\begin{definition}\label{def:t-multibase-seq}
Let $\Scal=(\bsx_n)_{n\ge 0}\subset[0,1)^d$ and let $\bsb=(b_1,\dots,b_d)$ with $b_j\ge 2$.
Let $\bst=(t_1,\dots,t_d)\in\NN_0^d$.
We call $\Scal$ a $(\bst,d)$-sequence in (multiple) base $\bsb$ if, for any
$r\in\NN_0$ and any $\bsk=(k_1,\dots,k_d)\in\NN_0^d$ with $k_j\ge t_j$,
every elementary $(k_1-t_1,\dots,k_d-t_d)$-interval in base $\bsb$ contains exactly
$\prod_{j=1}^d b_j^{t_j}$ points from 
$\{\bsx_{rB},\bsx_{rB+1},\dots,\bsx_{(r+1)B-1}\}$, where
$B:=\prod_{j=1}^d b_j^{k_j}$.
\end{definition}

We mainly focus on the case $\bst=\bszero:=(0,\dots,0)$.
As Tezuka pointed out~\cite{Tezuka2016trac}, this framework includes well-known sequences as special cases.
For example, the Halton sequence is a $(\bszero,d)$-sequence in base $(p_1,\dots,p_d)$, where $p_j$ denotes the $j$th prime number.
Any $(0,d)$-sequence in base $b$ (e.g., Faure sequences) is a $(\bszero,d)$-sequence in base $(b,\dots,b)$.
Furthermore, the notion of a $(0,\bse,d)$-sequence in base $b$ is identical to that of a $(\bszero,d)$-sequence in base $(b^{e_1},\dots,b^{e_d})$.
Therefore, Theorem~\ref{thm:Nied_is_0ed} implies that generalized Niederreiter sequences are $(\bszero,d)$-sequences in base $(b^{e_1},\dots,b^{e_d})$, where $e_j=\deg(p_j)$.

\subsection{Scrambling}\label{subsec:scramble}

Scrambling is a procedure that applies a random permutation to the digits of a point's coordinates,
aiming to preserve the equidistribution structure
of the underlying point set while enabling unbiased estimation and variance assessment.

\begin{definition}
Let $x = \xi_1 b^{-1}+\xi_2b^{-2}+\xi_3b^{-3}+\cdots$ with $\xi_k \in \Zb$. A \textit{nested scramble} in base $b$ is a random map $\sigma(x) = y$ that randomizes the digits $\xi_k$ to obtain $y = \eta_1 b^{-1}+\eta_2b^{-2}+\eta_3b^{-3}+\cdots$, where the digits $\eta_i$ are determined via permutations as
\[
\eta_1 = \pi_{}(\xi_1), \qquad
\eta_k = \pi_{ \xi_1 \xi_2 \dots \xi_{k-1}}(\xi_k) \quad (k \ge 2),
\]
where all $\pi$ and $\pi_{\xi_1 \xi_2 \dots \xi_{k-1}}$ are chosen from permutations on $\Zb$ and are mutually independent. For a point $\bsx = (x_1, \dots, x_d) \in [0,1)^d$, a nested scramble in base $(b_1, \dots, b_d)$ is a map $\sigma = (\sigma_1, \dots, \sigma_d)$ where each $\sigma_j$ is an independent random scramble in base $b_j$ applied to $x_j$.
\end{definition}

For a scramble to be effective in RQMC, it should satisfy certain uniformity properties.
These uniformity properties are crucial for variance analysis.

\begin{definition}
\begin{enuroman}
\item A random permutation $\pi$ on $\Zb$ has \textit{single-fold uniformity} if, for any $x \in \Zb$, $\pi(x)$ is uniformly distributed on $\Zb$.
\item A random permutation $\pi$ on $\Zb$ has \textit{two-fold uniformity} if, for any $x\neq y \in \Zb$, the pair $(\pi(x),\pi(y))$ is uniformly distributed over the $b(b-1)$-element set $\{(c,d) \in \Zb \times \Zb \mid c \neq d \}$.
\item A nested scramble has \textit{one-fold (resp.\ two-fold) uniformity} if all the permutations that constitute the scramble have one-fold (resp. two-fold) uniformity.
\end{enuroman}
\end{definition}

Note that the two-fold uniformity implies the one-fold uniformity.
Several methods exist to implement a scramble with these properties. Owen \cite{Ow95} introduced \textit{fully nested scrambling}, where the permutations are chosen independently and uniformly from the $b!$ possible permutations on $\mathbb{Z}_b$. This method is a conceptual ideal but can be costly both in storing the permutations and applying them. As a fast and practical alternative, the \textit{affine matrix scramble} has been proposed \cite{Ma98,Te94,Ow03}.
As in Section~\ref{sec:construction},
we fix a bijection $\phi \colon \Zb \to \Fb$ with $\phi(0)=0$.
Further, for a real number $x \in [0,1)$ with the base-$b$ expansion
$x = \xi_1 b^{-1} + \xi_2 b^{-2} + \cdots$,
where $\xi_k \in \{0, 1, \dots, b-1\}$, we denote the corresponding vector of its digits as
\begin{equation}\label{eq:real2vec}
\vec{x} := (\phi(\xi_{1}),\phi(\xi_{2}), \dots)^\top.
\end{equation}

\begin{definition}[Affine matrix scramble]
For $1 \leq j \leq d$, let $\Delta_j \in \mathbb{F}_b^{\infty}$ and $M_j \in \mathbb{F}_b^{\infty \times \infty}$. The families $\{\Delta_j\}$ and $\{M_j\}$ are independent, with $\Delta_j$ having i.i.d. uniform entries in $\mathbb{F}_b$. Each $M_j$ is a random lower triangular matrix with diagonal entries drawn uniformly from $\mathbb{F}_b \setminus \{0\}$ and sub-diagonal entries drawn uniformly from $\mathbb{F}_b$.
Let $\bsx = (x_1, \dots, x_d) \in [0,1)^d$ and let $\vec{x}_j$ be the digit vector of $x_j$ defined in \eqref{eq:real2vec}.
Then the randomly scrambled point of $\bsx$, denoted $\bsx'$, is given by
\[
\bsx' := \bigl(\psi(M_1 \vec{x}_1 + \Delta_1), \dots, \psi(M_d \vec{x}_d + \Delta_d)\bigr),
\]
where $\psi$ is defined in \eqref{eq:vec2val}.
\end{definition}

Both fully nested scrambling and affine matrix scrambling possess two-fold uniformity. A key consequence of such scrambles is that a single scrambled point $\sigma(\bsx)$ is uniformly distributed on $[0,1)^d$. Furthermore, scrambling preserves the equidistribution property of the entire sequence. The proof proceeds similarly to the case of $(t,d)$-sequences \cite{Ow95} and is therefore omitted.

\begin{proposition}
Let $\Scal = (\bsx_0, \bsx_1, \cdots)$ be a $(\bst,d)$-sequence in base $\bsb$, and let $\sigma$ be a nested scramble in base $\bsb$. Then the scrambled sequence $\Scal' := (\sigma(\bsx_0), \sigma(\bsx_1), \cdots)$ is a $(\bst,d)$-sequence in base $\bsb$ with probability $1$.
\end{proposition}

For sequences like the Sobol' and Niederreiter sequences, which are $(0,\bse,d)$-sequences in base $b$, their strong equidistribution property is defined with respect to the mixed base $(b^{e_1},\dots,b^{e_d})$. To preserve this specific structure, it is natural to apply scrambling in this mixed base. We refer to this as a \textit{coarse scramble}, 
a central concept of this paper, in contrast to the usual fine or digitwise scramble in base $(b,\dots,b)$.
This can be achieved by fully nested coarse scramble or by affine matrix scramble over a field $\mathbb{F}_{b^{e_j}}$ with mixed bases.
By a \emph{fully nested coarse scramble} in base $\bsb = (b^{e_1},\dots,b^{e_d})$, we mean a nested scramble in base $\bsb$ (Definition~2.10) in which, for each coordinate $j$, the permutations on $\mathbb{Z}_{b^{e_j}}$ are chosen independently and uniformly from the $(b^{e_j})!$ possible permutations (i.e., the mixed-base analogue of Owen's fully nested scrambling).
In this paper, we propose the block affine matrix scramble, which generalizes the affine matrix scramble to operate on blocks of digits and directly corresponds to performing operations in the mixed base.
\begin{definition}[Block affine matrix scramble]
For a prime power $b$, the \textit{block affine matrix scramble} in base $(b^{e_1}, \dots, b^{e_d})$ is defined as follows. For $1 \le j \le d$, let $\Delta_j \in \mathbb{F}_b^{\infty}$ and $M_j \in \mathbb{F}_b^{\infty \times \infty}$ be independent random vectors and matrices.
The entries of each $\Delta_j$ are i.i.d.\ and uniformly distributed over $\mathbb{F}_b$.
Each $M_j$ is a block lower-triangular matrix with block size $e_j$, where diagonal blocks are drawn uniformly from the nonsingular matrices in $\mathbb{F}_b^{e_j \times e_j}$, and off-diagonal blocks are drawn uniformly from all matrices in $\mathbb{F}_b^{e_j \times e_j}$. The scrambled point $\bsx'$ is then computed as
\[
\bsx' := \bigl(\psi(M_1 \vec{x}_1 + \Delta_1), \dots, \psi(M_d \vec{x}_d + \Delta_d)\bigr).
\]
Note that this recovers the usual affine matrix scramble when $e_j=1$ for all $j$.
\end{definition}

This generalized scramble retains the essential properties of the original. The following result is an analogue of the standard case \cite{Ma98, Ow03}, and its proof is omitted.

\begin{lemma}
For a prime power $b$, the block affine matrix scramble in base $\bsb = (b^{e_1}, \dots, b^{e_d})$ is a nested scramble in base $\bsb$ and has two-fold uniformity.
\end{lemma}
\subsection{Nested ANOVA decomposition and gain coefficients}\label{sec:gain-definition}

To analyze the variance of an RQMC estimator, we decompose it into contributions from the integrand function and the point set. The function's structure is captured by the nested ANOVA (Analysis of Variance) decomposition, which we introduce first.

The nested ANOVA decomposition breaks down a function $f: [0,1)^d \to \RR$ into a sum of orthogonal components. 
This decomposition is typically defined using an orthonormal basis, such as the Haar basis \cite{Ow97a} or Walsh basis \cite{DP10book};
the equivalence of our formulation to such definitions is established in Appendix~\ref{sec:anova-equiv}.
The decomposition is constructed with respect to the mixed base $\bsb$;
while this dependency is often omitted from the following notation for brevity, all resulting ANOVA components implicitly depend on $\bsb$.
The decomposition is given by
\begin{equation}\label{eq:ANOVA}
f(\bsx)
= f_{\emptyset} + \sum_{\emptyset \neq u \subseteq 1{:}d} \sum_{\bsk \in \NN_0^{|u|}} \beta_{u,\bsk}(\bsx_u),
\end{equation}
where $f_{\emptyset} = \int_{[0,1)^d} f(\bsx) \,d\bsx$, and the component $\beta_{u,\bsk}$ is defined via auxiliary functions $\tilde{\beta}_{u,v,\bsk}$ as
\begin{align}
\beta_{u,\bsk}(\bsx_u) &= \sum_{v \subseteq u} (-1)^{|u|-|v|} \tilde{\beta}_{u,v,\bsk}(\bsx_u), \label{eq:beta} \\
\tilde{\beta}_{u,v,\bsk}(\bsx_u)
&= m_{u,v,\bsk} \int_{\BOX_{u,v,\bsk}(\bsx_u)} d\bsy_u \int_{[0,1)^{d-|u|}} f(\bsy_u,\bsz_{-u}) d\bsz_{-u},\label{eq:beta_tilde}
\end{align}
where $\BOX_{u,v,\bsk}(\bsx_u)$ is the elementary interval containing $\bsx_u$ defined by
\begin{equation}\label{eq:BOX}
\BOX_{u,v,\bsk}(\bsx_u)
= \prod_{j \in v} \left[\frac{\lfloor x_j b_j^{k_j+1} \rfloor}{b_j^{k_j+1}}, \frac{\lfloor x_j b_j^{k_j+1} \rfloor+1}{b_j^{k_j+1}} \right)
\times \prod_{j \in u\setminus v} \left[\frac{\lfloor x_j b_j^{k_j} \rfloor}{b_j^{k_j}}, \frac{\lfloor x_j b_j^{k_j} \rfloor+1}{b_j^{k_j}} \right),
\end{equation}
and $(m_{u,v,\bsk})^{-1}$ is the volume of $\BOX_{u,v,\bsk}(\bsx_u)$, namely
\begin{equation}\label{eq:m_uvk}
m_{u,v,\bsk} := \prod_{j \in v} b_j^{k_j+1} \prod_{j \in u \setminus v} b_j^{k_j}.
\end{equation}
By definition, $\tilde{\beta}_{u,v,\bsk}$ is constant on each elementary interval $\BOX_{u,v,\bsk}(\bsx_u)$, taking the value of the average of $f$ over it. These components are related by the following identity:
\begin{equation} \label{eq:beta2}
\sum_{\substack{\bsk' \in \NN_0^{|u|} \\ k'_j \le k_j \, (\forall j \in u)}}\beta_{u,\bsk'}(\bsx_u) = \tilde{\beta}_{u,u,\bsk}(\bsx_u).
\end{equation}
This is also shown in Appendix~\ref{sec:anova-equiv}.
The components $\beta_{u,\bsk}$ are mutually orthogonal. By setting
\[
\sigma_{u,\bsk}^2 := \Var \beta_{u,\bsk},
\]
it follows from orthogonality (Parseval’s identity) that the total variance of the function is
\[
\Var f =
\sum_{\emptyset \neq u \subseteq 1{:}d}
\sum_{\bsk \in \NN_0^{|u|}}
\sigma_{u,\bsk}^2.
\]
From this, the variance of the standard Monte-Carlo (MC) estimator with $n$ independent random points is given by
\begin{equation}\label{eq:MC-variance}
\frac{\Var f}{n} = \dfrac{1}{n}\sum_{\emptyset \neq u \subseteq 1{:}d}
\sum_{\bsk \in \NN_0^{|u|}}
\sigma_{u,\bsk}^2.
\end{equation}

Next, we introduce a tool to characterize the structure of the point set. The \textit{gain coefficients} quantify how a specific $n$-point set $P_n$ is expected to perform for each ANOVA component when used in an RQMC estimate. The definition for base $b$ was given in \cite{Ow97a}, and its extension to mixed bases was introduced in \cite[Section~2.3]{OP24gain}. 

\begin{definition}\label{def:gain}
Let $P_n=\{\bsx_0,\ldots,\bsx_{n-1}\}\subset [0,1)^d$ be an $n$-element point set. For $\emptyset \neq u \subseteq 1{:}d$ and $\bsk=(k_j)_{j\in u}\in \NN_0^{|u|}$, we define the gain coefficient of $P_n$ in base $\bsb=(b_1,\dots,b_d)$ by
\begin{align} 
G_{u,\bsk}(P_n)
&:= \frac{1}{n} \prod_{j \in u} (b_j-1)^{-1} \cdot \widetilde{G}_{u,\bsk}(P_n), \label{eq:G_Gtilde}
\end{align}
where
\begin{align}
\widetilde{G}_{u,\bsk}(P_n)
:= \sum_{i=0}^{n-1}\sum_{i'=0}^{n-1} \prod_{j \in u} \left(b_j \chi(\lfloor b_j^{k_j+1}x_{i,j}\rfloor=\lfloor b_j^{k_j+1}x_{i',j}\rfloor)
- \chi(\lfloor b_j^{k_j}x_{i,j}\rfloor=\lfloor b_j^{k_j}x_{i',j}\rfloor) \right) \label{eq:Gtilda_def}.
\end{align}
\end{definition}

An explicit formula for the gain coefficients is known \cite[(2.8)]{OP24gain},
\begin{equation} \label{eq:Gtilda_formula}
\widetilde{G}_{u,\bsk}(P_n)
= \sum_{v \subseteq u} H_{u,v} C_{u,v,\bsk}(P_n),     
\end{equation}
where we use the notation
$\boldsymbol{k} + \boldsymbol{1}_v = (\ell_j)_{j \in u}$ where $\ell_j = k_j + 1$ if
$j \in v$ and $\ell_j = k_j$ if $j \in u \setminus v$, and
\begin{align}
H_{u,v} &=
(-1)^{|u|-|v|} m_{u,v,\bszero} = 
\prod_{j \in v} b_j \prod_{j \in u \setminus v} (-1), \label{eq:H_def}\\
C_{u,v,\bsk}(P_n) &= \sum_{i=0}^{n-1}\sum_{i'=0}^{n-1}
\chi(\text{$\bsx_i$ and $\bsx_{i'}$ are in the same elementary $(\bsk + \bsone_v)$-interval}). \label{eq:C_def}
\end{align}

With these two tools—the ANOVA decomposition for the function and the gain coefficients for the point set—we can now state the variance of the RQMC estimator \cite{Ow97a,OP24gain}.

\begin{theorem}
Let $P_n \subset [0,1)^d$ be an $n$-point set, and let $\tilde{P}_n$ be its randomly scrambled point set using a nested scramble with two-fold uniformity. Then the variance of the RQMC estimator $I(f; \tilde{P}_n) = n^{-1}\sum_{\bsx \in \tilde{P}_n} f(\bsx)$ is given by
\begin{equation}\label{eq:variance-gain}
\Var[I(f; \tilde{P}_n)] = \sum_{\emptyset \neq u \subseteq 1{:}d} \sum_{\bsk \in \NN_0^{|u|}} \dfrac{G_{u,\bsk}(P_n)}{n}\sigma_{u,\bsk}^2.
\end{equation}
\end{theorem}

By comparing the RQMC variance \eqref{eq:variance-gain} with the MC variance \eqref{eq:MC-variance}, it becomes clear that the gain coefficients $G_{u,\bsk}(P_n)$ measure how much the underlying point set $P_n$ reduces the variance compared to ordinary MC sampling for each frequency component $(u, \bsk)$ of the integrand. A smaller gain coefficient indicates a greater variance reduction for that component.

\section{Gain coefficients of scrambled points}\label{sec:gain}

\subsection{General results}
The theoretical properties of gain coefficients for scrambled $(\bszero,d)$-sequences can be established by generalizing the recent analysis of Owen and Pan for Halton sequences \cite{OP24gain}. The key insight enabling this generalization is that their foundational formula for the counting term $C_{u,v,\bsk}(P_n)$ (re-derived here as Lemma 3.1) does not depend on the specific structure of the Halton sequence, but only on the defining property of a $(\bszero,d)$-sequence. This allows us to extend their main results to our broader framework.

\begin{lemma}\label{lem:Cformula}
Let $\Scal = \{\bsx_0, \bsx_1, \dots\} \subset [0,1)^d$ be a $(\bszero,d)$-sequence in base $\bsb = (b_1, \dots, b_d)$.
Let $n \in \NN$
and $P_n = \{\bsx_{0}, \dots, \bsx_{n-1}\}$.
Then, for any $\emptyset \neq v \subseteq u \subseteq 1{:}d$ and $\bsk=(k_j)_{j\in u}\in \NN_0^{|u|}$, we have the following.
\begin{enuroman}
\item \label{item:Cformula1}
With $m_{u,v,\bsk}$ defined as in \eqref{eq:m_uvk}, we have
\begin{equation}\label{eq:Cformula}
C_{u,v,\bsk}(P_n)
= n + (2n-m_{u,v,\bsk})\left\lfloor \dfrac{n}{m_{u,v,\bsk}} \right\rfloor - m_{u,v,\bsk} \left\lfloor \dfrac{n}{m_{u,v,\bsk}} \right\rfloor^2.
\end{equation}
\item \label{item:Cformula2}
The values
$G_{u,\bsk}(P_n)$, $\widetilde{G}_{u,\bsk}(P_n)$ and $C_{u,v,\bsk}(P_n)$
are determined solely by $n$.
By a slight abuse of notation, we write
\begin{equation} \label{eq:abuse}
G_{u,\bsk}(n) := G_{u,\bsk}(P_n),
\qquad \widetilde{G}_{u,\bsk}(n) := \widetilde{G}_{u,\bsk}(P_n),
\qquad C_{u,v,\bsk}(n) := C_{u,v,\bsk}(P_n).
\end{equation}
\end{enuroman}
\end{lemma}

\begin{proof}
Since \eqref{item:Cformula2} follows from 
\eqref{eq:G_Gtilde}, \eqref{eq:Gtilda_formula} and \eqref{eq:Cformula},
it suffices to show $\eqref{eq:Cformula}$.
Let $n = qm_{u,v,\bsk} + r$ with $0 \le r < m_{u,v,\bsk}$
and let $\bsk' := \bsk + \bsone_v$.
The number of the elementary $\bsk'$-intervals in base $\bsb$
is $m_{u,v,\bsk}$.
From the definition of $(\bszero,d)$-sequence in base $\bsb$,
among such intervals,
$r$ intervals contain $q+1$ points from $P$
and the other $m_{u,v,\bsk} - r$ intervals contain $q$ points.
Thus \eqref{eq:C_def} implies that
\[
C_{u,v,\bsk}(P_n)
= r(q+1)^2 + (m_{u,v,\bsk} - r)q^2.
\]
After simplification using $r=n-qm_{u,v,\bsk}$ and $q = \lfloor n/m_{u,v,\bsk} \rfloor$,
we obtain \eqref{eq:Cformula}.
\end{proof}

This observation enables us to generalize results in \cite{OP24gain}, as follows.

\begin{theorem}\label{thm:mimic}
Let $\Scal = \{\bsx_0, \bsx_1, \dots\} \subset [0,1)^d$ be a $(\bszero,d)$-sequence in base $\bsb = (b_1, \dots, b_d)$.
Let $\emptyset \neq u \subseteq 1{:}d$ and $\bsk \in \NN_0^{|u|}$.
Then the gain coefficients in base $\bsb$ satisfy the following.
\begin{enuroman}
\item \label{item:mimic1}
If $1 \le n \le m_{u,\emptyset,\bsk}$,
\[
G_{u,\bsk}(n) = 1.
\]
\item \label{item:mimic2}
If $n = q m_{u,u,\bsk}$ for $q \in \NN$, then
\[
G_{u,\bsk}(n) = 0.
\]
\item \label{item:mimic3}
Let $n = q m_{u,u,\bsk} + r$ for $q,r \in \NN$ with
$0 < r < m_{u,u,\bsk}$. Then
\[
G_{u,\bsk}(n) = \dfrac{r}{n}G_{u,\bsk}(r).
\]
\item \label{item:mimic4}
For $j \in u$, define $\bsk'$ by $k'_j = k_j+1$ and $k'_\ell = k_\ell$ for $\ell \in u \setminus \{j\}$.
Then 
\[
G_{u,\bsk'}(nb_j) = G_{u,\bsk}(n).
\]
\item \label{item:mimic5}
For $n \in \NN$, we have
\[
G_{u,\bsk}\left(n \prod_{j \in u} b_j\right) = G_{u,\bszero}(n).
\]
\item \label{item:mimic6}
For any $\emptyset \neq v \subseteq u$,
\[
\sup_{n \in \NN} G_{v,\bszero}(n) \le \sup_{n \in \NN} G_{u,\bszero}(n)
\]
\item \label{item:mimic7}
Let $j_m = \argmin_{j \in u} b_j$. Then
\[
\sup_{n \in \NN} G_{u,\bszero}(n)
= \sup_{n \in \NN} G_{u,\bsk}(n)
\le \prod_{j \in u \setminus \{j_m\}  } \dfrac{b_j}{b_j-1}.
\]
\end{enuroman}
\end{theorem}

\begin{proof}
Except for items \eqref{item:mimic1} and \eqref{item:mimic3},  
the arguments for the corresponding theorems in \cite{OP24gain} carry over verbatim,  
since the proofs rely only on \eqref{eq:Gtilda_formula} and \eqref{eq:Cformula}.  
The correspondence between our results and those of \cite{OP24gain} is as follows:  
\eqref{item:mimic1} to Proposition~3.1 in \cite{OP24gain},
\eqref{item:mimic2} to Proposition~3.2,
\eqref{item:mimic3} to Proposition~3.3,
\eqref{item:mimic4} to Proposition~3.5,
\eqref{item:mimic5} to Corollary~3.6,
\eqref{item:mimic6} to Theorem~5.2,
\eqref{item:mimic7} to Theorem~5.3.

For \eqref{item:mimic1}, our result extends the corresponding proposition in \cite{OP24gain} to include the case $n = m_{u,\emptyset,\bsk}$.
Since $C_{u,v,k}(n) = n$ holds for this value of $n$, the original proof applies directly.  

For \eqref{item:mimic3}, since the corresponding proposition explicitly uses a property of the Halton sequence,  
a separate proof is required in our more general setting.
From \eqref{eq:Cformula},
since $m_{u,u,\bsk}$ is a multiple of $m_{u,v,\bsk}$,
for $n = qm_{u,u,\bsk}+r$ we have
\begin{align*}
C_{u,v,\bsk}(n)
&= n + (2n-m_{u,v,\bsk})\left(\dfrac{qm_{u,u,\bsk}}{m_{u,v,\bsk}} + \left\lfloor \dfrac{r}{m_{u,v,\bsk}} \right\rfloor \right) - m_{u,v,\bsk} \left(\dfrac{qm_{u,u,\bsk}}{m_{u,v,\bsk}} + \left\lfloor \dfrac{r}{m_{u,v,\bsk}} \right\rfloor \right)^2\\
&= \dfrac{q^2 m_{u,u,\bsk}^2 + 2rq m_{u,u,\bsk}}{m_{u,v,\bsk}} + C_{u,v,\bsk}(r).
\end{align*}
Hence
\begin{align*}
\widetilde{G}_{u,\bsk}(n)
= \sum_{v \subseteq u} H_{u,v} C_{u,v,\bsk}(n)
&= \sum_{v \subseteq u} H_{u,v} \dfrac{q^2 m_{u,u,\bsk}^2 + 2r m_{u,u,\bsk}}{m_{u,v,\bsk}} + \sum_{v \subseteq u} H_{u,v}C_{u,v,\bsk}(r)\\
&= (q^2 m_{u,u,\bsk}^2 + 2r m_{u,u,\bsk}) \sum_{v \subseteq u} (-1)^{|u|-|v|} + \widetilde{G}_{u,\bsk}(r)\\
&= (q^2 m_{u,u,\bsk}^2 + 2r m_{u,u,\bsk}) \prod_{j \in u}(1-1) + \widetilde{G}_{u,\bsk}(r)\\
&= \widetilde{G}_{u,\bsk}(r).
\end{align*}
Thus the item \eqref{item:mimic3} follows by the normalization of \eqref{eq:G_Gtilde}.
\end{proof}

As a corollary, the variance of the RQMC error using a scrambled $(\bszero,d)$-sequence decays as \(o(1/n)\),
which is strictly better than plain Monte Carlo.  
The proof is identical to that of \cite[Corollary~3.4]{OP24gain},  
with \(\bsx_i\) replaced by \(\tilde{\bsx}_i\).

\begin{corollary}
Let $\Scal \subset [0,1)^d$ be a $(\bszero,d)$-sequence in base $\bsb = (b_1, \dots, b_d)$,
and let $(\tilde{\bsx}_i)_{i\ge 0}$ 
be a sequence obtained by a random scrambling of $\Scal$ in base $\bsb$ that ensures two-fold uniformity.
Then, for any $f \in L_2([0,1)^d)$, we have
\[
\lim_{n \to \infty} n \cdot \Var \left(\dfrac{1}{n}\sum_{i=0}^{n-1}f(\tilde{\bsx}_i)\right) = 0.
\]
\end{corollary}

\subsection{Gain coefficients of scrambled \texorpdfstring{$(0,\bse,d)$}{(0,e,d)}-sequences}

We now specialize these general results to scrambled $(0,\bse,d)$-sequences, in particular Sobol' and Niederreiter sequences.
Since $(0,\bse,d)$-sequences in base $b$ are $(\bszero,d)$-sequences in base $(b^{e_1}, \dots, b^{e_d})$,
we will state the results for the latter.

First we consider non-asymptotic results for $n=b^m$.
\begin{corollary}\label{cor:gain-0ed}
Let $\Scal$ be a $(\bszero,d)$-sequence in base $\bsb=(b^{e_1}, \dots, b^{e_d})$.
Let $\emptyset \neq u \subseteq 1{:}d$, $\bsk \in \NN_0^{|u|}$ and
$j_m = \argmin_{j \in u} e_j$.
Then, for $n \in \NN_0$ we have
\[
G_{u,\bsk}(n) \le \Gamma_u,
\]
where
\begin{equation}\label{eq:Gamma_u}
\Gamma_u := \prod_{j \in u \setminus \{j_m\}} \frac{b^{e_j}}{b^{e_j}-1},
\end{equation}
with the convention that the empty product equals $1$.
For the case $n=b^m$ with $m \in \NN_0$, we have
\[
\begin{cases}
G_{u,\bsk}(b^m) = 0 & \text{if } \sum_{j\in u} e_j(k_j+1) \le m, \\
G_{u,\bsk}(b^m) \le \Gamma_u  & \text{if } \sum_{j\in u} e_j k_j < m < \sum_{j \in u} e_j(k_j+1), \\
G_{u,\bsk}(b^m) = 1 & \text{if } \sum_{j\in u} e_j k_j \ge m.
\end{cases}
\]
\end{corollary}

\begin{proof}
We use Theorem~\ref{thm:mimic}.
In our case, we have
$m_{u,\emptyset,\bsk} = \prod_{j \in u} b^{e_jk_j}$ and
$m_{u,u,\bsk} = \prod_{j \in u} b^{e_j(k_j+1)}$.
Thus the claim for general $n$ follows from \eqref{item:mimic7},
and the three claims for $n=b^m$ follow from
\eqref{item:mimic2}, \eqref{item:mimic7}, and \eqref{item:mimic1} of Theorem~\ref{thm:mimic}, respectively.
\end{proof}

\begin{remark}
By letting $\bse = (1,\dots,1)$,
we obtain the gain coefficients of $(0,d)$-sequence in base $b$ as 
\[
\sup_{n \in \NN} G_{u,\bsk}(n) \leq \left(\dfrac{b}{b-1}\right)^{|u|-1}.
\]
For $n=b^m$, this bound for $(0,m,d)$-nets is given in \cite{Ow97a}.
\end{remark}

Indeed, the exact maximum gain coincides with the upper bound given in Corollary~\ref{cor:gain-0ed}.

\begin{corollary}\label{cor:maxgain_uk}
Let $\Scal$ be a $(0,\bse,d)$-sequence in base $b$.
Let $\Gamma_u$ be constants defined in \eqref{eq:Gamma_u}.
Then, for any $\emptyset \neq u \subseteq 1{:}d$ and $\bsk \in \NN_0^{|u|}$ we have
\[
\max_{n \in \NN} G_{u,\bsk}(n) = \Gamma_u.
\]
\end{corollary}

\begin{proof}
According to Theorem~\ref{thm:mimic} \eqref{item:mimic7}, it suffices to show the theorem for $\bsk = \bszero$ and in particular that $n^* := \prod_{j \in u \setminus \{j_m\}} b^{e_j}$ attains the maximum.
Since $n^*$ is a multiple of $m_{u,v,\bszero}$ for $v \subsetneq u$,
\eqref{eq:Cformula} implies that
\[
C_{u,v,\bszero}(n^*)
= n^* + (2n^*-m_{u,v,\bszero})\dfrac{n^*}{m_{u,v,\bszero}} - m_{u,v,\bszero} \left( \dfrac{n^*}{m_{u,v,\bszero}} \right)^2
= \dfrac{(n^*)^2}{m_{u,v,\bszero}}
\]
for $v \subsetneq u$, and for $v=u$ we have $C_{u,u,\bszero}(n^*) = n^*$.
Thus it follows from \eqref{eq:Gtilda_formula} that
\begin{align*}
\widetilde{G}_{u,\bszero}(n^*)
= \sum_{v \subseteq u} H_{u,v} C_{u,v,\bszero}(n^*)
&= (n^*)^2\sum_{v \subsetneq u} (-1)^{|u|-|v|} + H_{u,u}n^*\\
&= (n^*)^2 \sum_{v \subseteq u} (-1)^{|u|-|v|} + H_{u,u}n^* - (n^*)^2
= H_{u,u}n^* - (n^*)^2.
\end{align*}
In conjunction with the normalization \eqref{eq:G_Gtilde}, we obtain the desired result.
\end{proof}

From Corollary~\ref{cor:gain-0ed} and Eq.~\eqref{eq:variance-gain},
we have the bound on the variance of the scrambled points estimator.
\begin{corollary}
Let $\Scal$ be a $(0,\bse,d)$-sequence in base $b$,
and let $\tilde{\bsx}_0, \tilde{\bsx}_1, \tilde{\bsx}_2, \dots$ be a randomly scrambled sequence of $\Scal$ in base $\bsb=(b^{e_1}, \dots, b^{e_d})$.
Let $\Gamma_u$ be constants defined in \eqref{eq:Gamma_u}.
Then, for $m \in \NN$ we have
\[
\Var \left(\dfrac{1}{b^m}\sum_{i=0}^{b^m-1}f(\tilde{\bsx}_i) \right)
\le \dfrac{1}{b^m} \sum_{\emptyset \neq u \subseteq 1{:}d}  \Gamma_u \sum_{\substack{\bsk\in \NN_0^{|u|}\\ \sum_{j \in u} e_j(k_j+1) > m}}\sigma_{u,\bsk}^2.
\]
\end{corollary}

Lastly, we show a slow growth of the maximal gain coefficients of the coarsely scrambled Sobol' or Niederreiter sequences.

\begin{corollary}
Let $\Scal$ be the $d$-dimensional Sobol' sequence (in base $b=2$), or the full Niederreiter sequence in prime power base $b$.
Then the gain coefficients for the coarse scramble are $O(\log d)$, specifically,
\[
\Gamma_d := \max_{\emptyset \neq u \subseteq 1{:}d} \sup_{\bsk \in \NN_0^{|u|}} \sup_{n \in \NN} G_{u,\bsk}(n)
\le e \lceil \log_b d + \log_b \log_b(d+b) + 2 \rceil.
\]
\end{corollary}

\begin{proof}
Because Sobol' sequences use only primitive polynomials, the degrees of their generators are never smaller than those of full Niederreiter sequences.
Consequently, $\Gamma_d$ for Sobol' sequences does not exceed that of full Niederreiter sequences, 
so it suffices to establish the result for the latter.

Let $I_b(n)$ be the number of monic irreducible polynomials in $\Fb$ of degree $n$,
and $N_b(n) := I_b(1) + \cdots + I_b(n)$.
It is well known (see, e.g., \cite[Theorem~3.1.4]{mullen2013handbook}) that, for $n \ge 2$,
\begin{equation*}\label{eq:upperbound_Ib}
nI_b(n) \le  b^n - 1.
\end{equation*}
Using this, $I_b(1) = b$, Corollary~\ref{cor:maxgain_uk}, and $\log(1+x) \le x$,
the gain coefficients of the $N_b(M)$-dimensional Niederreiter sequences, i.e., those using irreducible polynomials of degree up to $M$, are bounded by
\begin{align*}
\log \Gamma_{N_b(M)}
&= \sum_{n=1}^M \left(I_b(n) - \chi(n=1)\right) \log \left(1 + \dfrac{1}{b^n-1} \right) \\
&\le \sum_{n=1}^M \dfrac{b^n-1}{n} \dfrac{1}{b^n-1}
= \sum_{n=1}^M \dfrac{1}{n}
\le 1 + \int_{1}^M \dfrac{1}{x} \,dx = 1 + \log M.
\end{align*}
Let $p_1(x), p_2(x), \dots$ be all monic irreducible polynomials over $\Fb$,
listed in order of nondecreasing degree.
From \cite[Lemma~2]{Wa03}, for any $d \in \NN$ we have
\[
\deg(p_d(x)) \le \log_b d + \log_b \log_b(d+b) + 2,
\]
and thus $d \le N_b(M^*)$ holds with $M^* = \lceil \log_b d + \log_b \log_b(d+b) + 2 \rceil$.
Hence
\[
\Gamma_d
\le \Gamma_{N_b(M^*)}
\le \exp(1 + \log M^*)
= e \lceil \log_b d + \log_b \log_b(d+b) + 2 \rceil.
\]
Thus we are done.
\end{proof}

This logarithmic growth of the maximal gain coefficient stands in stark contrast to the exponential growth observed for standard Owen scrambling of Sobol' sequences, highlighting the primary theoretical advantage of coarse scrambling in high dimensions for the worst case.

\section{Comparison of coarse and usual scrambling}\label{sec:comparison}
In this section, we compare the coarse scrambling and the usual scrambling in base $b$.
We denote the two bases of interest as the coarse base,
$\bsb^{(\bcoarse)} := (b^{e_1}, \dots, b^{e_d}),$
and the usual base,
$\bsb^{(\busual)} := (b, \dots, b).$
Hereafter, superscripts or subscripts C and U will refer to quantities related to the coarse and usual bases, respectively.
We first give a relation between the nested ANOVA decomposition in coarse and usual bases.

\begin{lemma}\label{lem:beta-basechange}
Let $b,e_1,\dots,e_d \in \NN$ with $b \ge 2$.
Let $f \colon [0,1)^d \to \RR$ be a function, and
$\beta_{u,\bsk}^{(\bcoarse)}$ and 
$\beta_{u,\bsk}^{(\busual)}$ be the components of the nested ANOVA decomposition of $f$ in base $\bsb^{(\bcoarse)}$ and $\bsb^{(\busual)}$ as given in \eqref{eq:ANOVA}.
Then, for any $\emptyset \neq u \subseteq 1{:}d$ and $\bsk \in \NN_0^{|u|}$ we have
\begin{equation}\label{eq:beta_basechange}
\beta_{u,\bsk}^{(\bcoarse)}(\bsx) = \sum_{\substack{\bsl=(l_j) \in \NN_0^{|u|} \\ e_jk_j \le l_j < e_j(k_j+1) \, (\forall j \in u)}} \beta_{u,\bsl}^{(\busual)}(\bsx)
\end{equation}
and thus
\begin{equation}\label{eq:sigma_basechange}
(\sigma_{u,\bsk}^{(\bcoarse)})^2 = \sum_{\substack{\bsl=(l_j) \in \NN_0^{|u|} \\ e_jk_j \le l_j < e_j(k_j+1) \, (\forall j \in u)}} (\sigma_{u,\bsl}^{(\busual)})^2.
\end{equation}
\end{lemma}

\begin{proof}
The definition of $\BOX$ \eqref{eq:BOX} implies that, for any $\bsx \in [0,1)^{|u|}$
\[
\BOX_{u,u,(e_j(k_j+\chi(j \in v))-1)_j}^{(\busual)}(\bsx)
= 
\BOX_{u,v,\bsk}^{(\bcoarse)}(\bsx)
\]
and thus
\[
\tilde{\beta}_{u,u,(e_j(k_j+\chi(j \in v)-1))_j}^{(\busual)}(\bsx)
= \tilde{\beta}_{u,v,\bsk}^{(\bcoarse)}(\bsx).
\]
Then, by inclusion-exclusion principle, \eqref{eq:beta2}, 
the above equation, and Eq.~\eqref{eq:beta},
we have
\begin{align*}
\sum_{\substack{\bsl=(l_j) \in \NN_0^{|u|} \\ e_jk_j \le l_j < e_j(k_j+1) \, (\forall j \in u)}} \beta_{u,\bsl}^{(\busual)}(\bsx)
&= \sum_{v \subseteq u} (-1)^{|u|-|v|}
\sum_{\substack{\bsl=(l_j) \in \NN_0^{|u|} \\ l_j < e_j(k_j+\chi(j \in v)) \, (\forall j \in u)}} \beta_{u,\bsl}^{(\busual)}(\bsx)\\
&= \sum_{v \subseteq u} (-1)^{|u|-|v|}
\tilde{\beta}_{u,u,(e_j(k_j+\chi(j \in v))-1)_j}^{(\busual)}(\bsx)\\
&= \sum_{v \subseteq u} (-1)^{|u|-|v|} \tilde{\beta}_{u,v,\bsk}^{(\bcoarse)}(\bsx_u)\\
&= \beta_{u,\bsk}^{(\bcoarse)}(\bsx_u),
\end{align*}
Thus we have the first statement.
The second one follows from the fact that $\beta$'s are orthogonal.
\end{proof}

Hereafter, let $\Scal$ be a $(0,\bse,d)$-sequence in base $b$.
We denote the variance of the estimator using points scrambled in the coarse and usual bases by $\sigma^2_C(n)$ and $\sigma^2_U(n)$, respectively:
\begin{align}
    \Vcoarse(n) &:= \Var\left(\frac{1}{n}\sum_{i=0}^{n-1}f(\tilde{\bsx}_i^{(C)})\right), \label{eq:var_coarse} \\
    \Vusual(n) &:= \Var\left(\frac{1}{n}\sum_{i=0}^{n-1}f(\tilde{\bsx}_i^{(U)})\right). \label{eq:var_usual}
\end{align}
Lemma~\ref{lem:beta-basechange} implies that 
\begin{equation} \label{eq:sigma-seq-basechange}
\Vcoarse(n)
= \frac{1}{n}\sum_{\emptyset \neq u \subseteq 1{:}d} \sum_{\bsk \in \NN_0^{|u|}} G_{u,\bsk}^{(\bcoarse)}(n) (\sigma_{u,\bsk}^{(\bcoarse)})^2
= \frac{1}{n}\sum_{\emptyset \neq u \subseteq 1{:}d} \sum_{\bsl \in \NN_0^{|u|}} G_{u,(\lfloor l_j/e_j \rfloor)_{j \in u}}^{(\bcoarse)}(n) (\sigma_{u,\bsl}^{(\busual)})^2.    
\end{equation}
Furthermore, for $n=b^m$,
Lemma~\ref{lem:beta-basechange} combining with Corollary~\ref{cor:gain-0ed} and the inequality
\[
e_j(\lfloor l_j/e_j \rfloor + 1)
\le e_j(l_j/e_j + 1)
= l_j + e_j
\]
implies that
\begin{align}
\Vcoarse(b^m)
&= \frac{1}{b^m}\sum_{\emptyset \neq u \subseteq 1{:}d} \sum_{\substack{\bsk \in \NN_0^{|u|} \\ \sum_{j \in u} e_j(k_j+1) > m}} G_{u,\bsk}^{(\bcoarse)}(b^m) (\sigma_{u,\bsk}^{(\bcoarse)})^2 \notag\\
&= \frac{1}{b^m}\sum_{\emptyset \neq u \subseteq 1{:}d} \sum_{\substack{\bsl \in \NN_0^{|u|} \\ \sum_{j \in u} e_j(\lfloor l_j/e_j \rfloor + 1)  > m}} G_{u,(\lfloor l_j/e_j \rfloor)_{j \in u}}^{(\bcoarse)}(b^m) (\sigma_{u,\bsl}^{(\busual)})^2 \label{eq:general_bound_coarse_complex}\\
&\le \frac{1}{b^m}\sum_{\emptyset \neq u \subseteq 1{:}d} \Gamma_u \sum_{\substack{\bsl \in \NN_0^{|u|} \\ \sum_{j \in u} l_j > m - \sum_{j \in u} e_j}}  (\sigma_{u,\bsl}^{(\busual)})^2, \label{eq:general_bound_coarse}
\end{align}
where $\Gamma_u$ is defined as in \eqref{eq:Gamma_u}.
Let us compare \eqref{eq:general_bound_coarse} with the general bound for the usual scramble of the (digital) $(t,m,d)$-net (see \cite{PO21, GS22}) as
\begin{equation}\label{eq:general_bound_dense}
\Vusual(b^m)
\le \frac{1}{b^m}\sum_{\emptyset \neq u \subseteq 1{:}d} b^{t_u}\left(\dfrac{b}{b-1}\right)^{|u|-1} \sum_{\substack{\bsl \in \NN_0^{|u|} \\ \sum_{j \in u} l_j > m - t_u}} (\sigma_{u,\bsl}^{(\busual)})^2, 
\end{equation}
where $t_u$ denotes the $t$-value of the projection of the digital sequence $\Scal$ onto the coordinates in $u$.  
In our case, by Theorem~\ref{thm:t-val-nied} we can take $t_u = \sum_{j \in u}(e_j-1)$.
It follows that the bound \eqref{eq:general_bound_coarse} has more terms, 
but each with a smaller constant factor.  
We can formalize this observation in the following ``meta'' theorem.
This can be applied to \cite[Theorem~2]{Ow98}
and we obtain a corresponding corollary.

\begin{theorem}
Let $\Scal$ be a $(0,\bse,d)$-sequence in base b. The variance bound for coarse scrambling \eqref{eq:general_bound_coarse} is strictly tighter than the standard variance bound for usual scrambling \eqref{eq:general_bound_dense} applied to a hypothetical digital $(t,m,s)$-net with
$t_u = \sum_{j \in u} e_j$.
A direct consequence is that any convergence rate proven for usual scrambling that relies solely on the general bound of the form \eqref{eq:general_bound_dense} also holds for coarse scrambling.
\end{theorem}

\begin{corollary}
Let $\Scal$ be a $(0,\bse,d)$-sequence in base $b$.
Assume that $\partial^df/(\partial x_1 \cdots \partial x_d)$ exists and it is Lipschitz  continuous.
Then, the variance $\Vcoarse(n)$ of the estimator based on coarse scrambling with $n=b^m$,
given in \eqref{eq:var_coarse},
decays as $O(n^{-3} (\log n)^{d-1})$.
\end{corollary}

Let us proceed more detailed comparison.
Comparing \eqref{eq:general_bound_coarse} and \eqref{eq:general_bound_dense},
The variance $\Vcoarse(b^m)$ of the estimator based on coarse scrambling may be larger for subsets $u$ with small $t_u$ and small cardinality $|u|$.  
In some cases, it can be shown that the precise $t$-value satisfies 
$t_u = \sum_{j \in u} (e_j-1)$ \cite{dick2008exact}, 
although this does not hold for all subsets.  

The most important case is the one-dimensional projection, i.e., $u = \{j\}$.  
For such one-dimensional projections, the Sobol' sequence (or, more generally, Niederreiter sequences 
with non-singular upper-triangular generating matrices) forms a $(0,1)$-sequence in base $b$, i.e., $t_{\{j\}} = 0$.
Thus the variance for the $\{j\}$-component for the usual scramble is given by
\begin{equation}\label{eq:1dim-dense}
\sum_{l \ge m} (\sigma_{\{j\},l}^{(\busual)})^2.
\end{equation}
On the other hand, for the coarse scramble, 
Corollary~\ref{cor:gain-0ed}
and \eqref{eq:general_bound_coarse_complex} imply that
the variance for the $\{j\}$-component is bounded by
\begin{equation}\label{eq:1dim-coarse}
\sum_{(\lfloor l/e_j \rfloor + 1)e_j > m} G_{\{j\},(\lfloor l_j/e_j \rfloor)}^{(\bcoarse)}(b^m) (\sigma_{\{j\},l}^{(\busual)})^2
\le \sum_{l \ge \lfloor m/e_j \rfloor e_j}  (\sigma_{\{j\},l}^{(\busual)})^2
\end{equation}
Thus, since $l \ge m$ implies $l \ge \lfloor m/e_j \rfloor e_j$,
the exact variance for the usual scramble in \eqref{eq:1dim-dense} is no greater than
the bound for the coarse scramble in \eqref{eq:1dim-coarse}.
This means that, for one-dimensional projections, the coarse scramble can be significantly worse than the usual one.
To avoid this drawback, we may choose $m$ to be a multiple of $e_j$,
in which case the two bounds coincide.
This avoidance is observed in our numerical experiment.

In practical applications, such as those in finance, integrands often have what is broadly termed a low effective dimension \cite{Lbook}. This notion is typically specified in two main ways. The first is having a \textit{low superposition dimension}, where the variance is dominated by interactions among a small number of variables (small $|u|$). The second is having a \textit{low truncation dimension}, where the variance is concentrated in the first few coordinates.
As our analysis shows, the trade-off between coarse and usual scrambling depends on the integrand's effective dimension. For functions with a low superposition dimension, where low-order interactions (especially 1D projections) dominate, usual scrambling is superior due to its perfect $t=0$ property on 1D components.
For functions with a low truncation dimension in high-dimensional settings, where the influence of higher-order interactions is non-negligible, the superior $O(\log d)$ scaling of the gain coefficient $\Gamma_u$ for coarse scrambling can become the decisive advantage, making it a compelling alternative.

\section{Numerical Experiments}\label{sec:experiments}

To empirically validate our theoretical findings, we conduct three numerical integration experiments comparing the performance of coarse scrambling with usual scrambling on Sobol' sequences.

\subsection{Experimental Setup}

For both scrambling strategies, we estimate the integral of test functions using a RQMC estimator.
In all experiments involving Sobol' sequences, we use the construction of Joe and Kuo \cite{JK08}.
The number of sample points is set to $n=2^m$, where $m$ ranges from 1 to 16. To assess the estimator's stability, we compute the root mean square error (RMSE) over $R=200$ independent replications for each value of $n$. For sufficiently smooth integrands, theory predicts a convergence rate of approximately $O(n^{-1.5})$ for the RMSE. The usual and coarse scrambling methods are implemented using affine matrix scramble and block affine matrix scramble, respectively.

We use three distinct test functions designed to probe different aspects of the scrambling methods:
\begin{enumerate}
    \item A \textbf{37-dimensional linear function}. Its variance is entirely contained in one-dimensional projections ($|u|=1$), giving it a low superposition dimension.

    \item A \textbf{100-dimensional weighted smooth function}. This function is designed to have a low truncation dimension.

    \item \textbf{A 19-dimensional function unfavorable to usual scrambling}. 
    This function is designed so that the usual scrambled Sobol' estimator exhibits an unusually large gain coefficients at the sample size considered.
\end{enumerate}
The results, presented as log-log plots of the RMSE against the number of points $n$, are shown in  Figures~\ref{fig:RQMC_linear}--\ref{fig:RQMC_anti}.

\begin{figure}[ht]
    \centering
    \includegraphics[width=\linewidth]{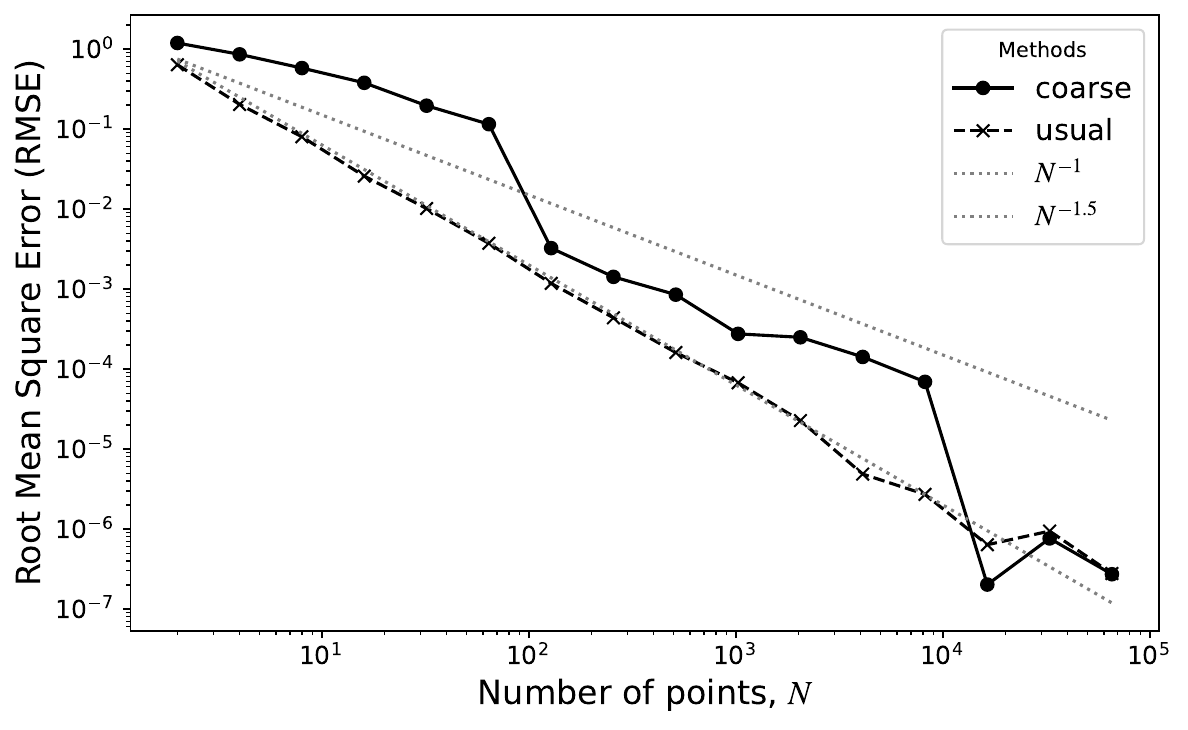}
    \caption{Log-log plot of the number of points vs.\ the root mean square error for coarse (block affine) and usual affine scrambling for the linear function ($d=37$). Reference lines for $O(n^{-1})$ and $O(n^{-1.5})$ are included.}
    \label{fig:RQMC_linear}
\end{figure}

\begin{figure}[ht]
    \centering
    \includegraphics[width=\linewidth]{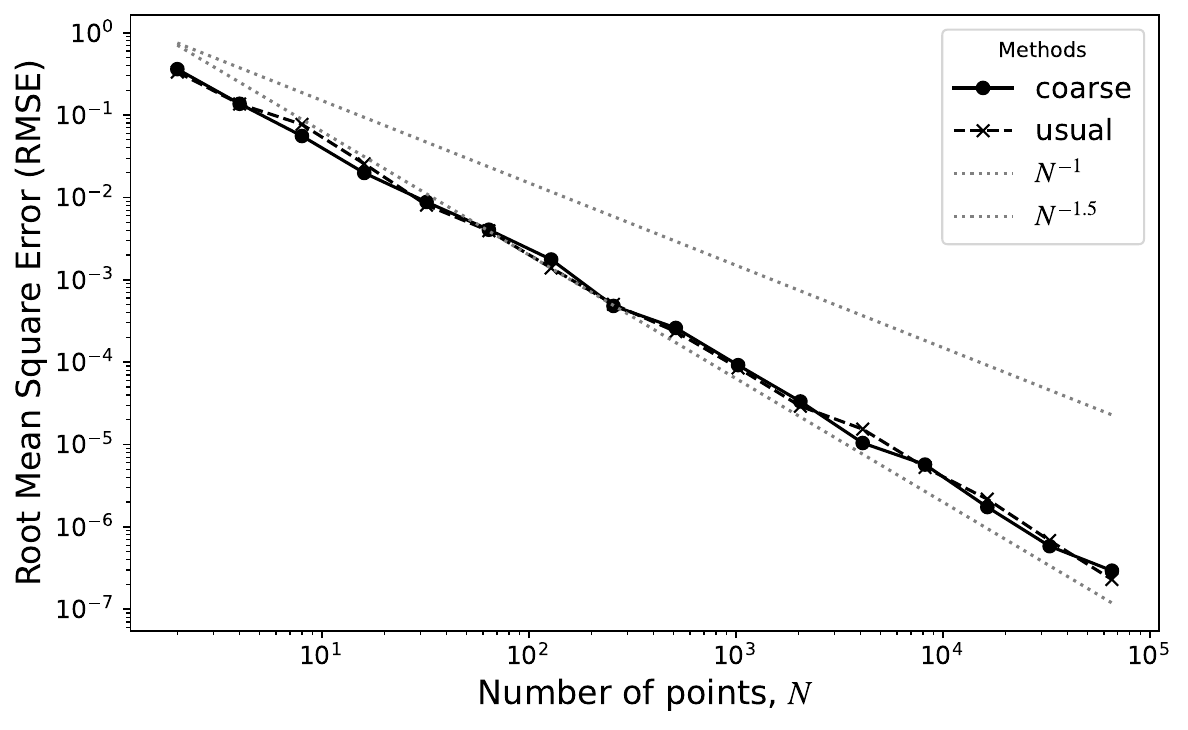}
    \caption{Log-log plot of the number of points vs.\ the root mean square error for coarse (block affine) and usual affine scrambling for the weighted smooth function ($d=100$). Reference lines for $O(n^{-1})$ and $O(n^{-1.5})$ are included.}
    \label{fig:RQMC_weighted}
\end{figure}

\begin{figure}[ht]
    \centering
    \includegraphics[width=\linewidth]{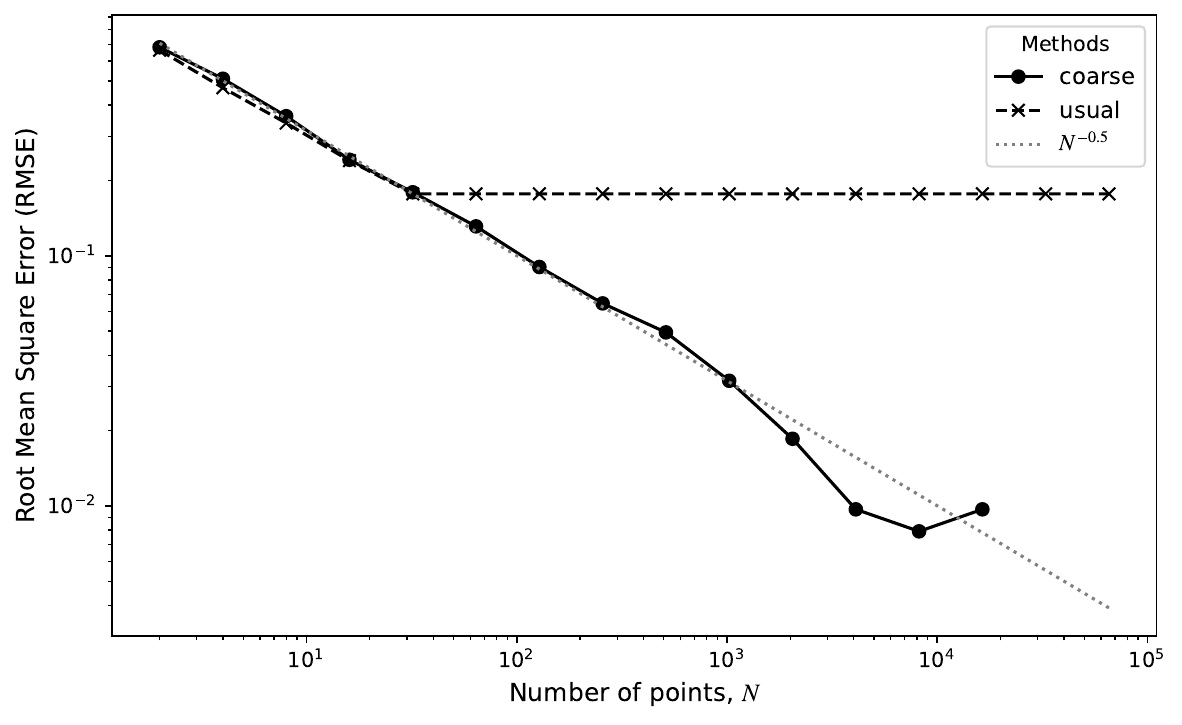}
    \caption{Log-log plot of the number of points vs.\ the root mean square error for coarse (block affine) and usual affine scrambling for the 19-dimensional function unfavorable to usual scrambling. A reference line for $O(n^{-0.5})$ is included. For $n=2^{15}$ and $2^{16}$, the RMSE for coarse scrambling is omitted, since over $R=200$ independent trials the estimator always equaled the exact integral value.}
    \label{fig:RQMC_anti}
\end{figure}

\subsection{Results and Discussion}

\subsubsection{Linear Function}
Our first test integrand is the 37-dimensional linear function:
\[
f(\bsx) = \sum_{j=1}^{37} x_j.
\]
The dimension $d=37$ was chosen because for the Sobol' sequence, the corresponding generator polynomial degrees result in block sizes $e_j$ up to 7 (see Table~\ref{table:degree}).

As shown in Figure~\ref{fig:RQMC_linear}, the results align perfectly with the theoretical analysis in Section~\ref{sec:comparison}. The usual scrambling method consistently outperforms coarse scrambling, exhibiting a smooth convergence rate close to $O(n^{-1.5})$. In contrast, the coarse scrambling method shows a stepped convergence.
The error convergence follows a baseline rate of approximately $O(n^{-1})$,
superimposed with significant drops when $m$ is a multiple of 7 (i.e., at $n=2^7$ and $n=2^{14}$). This confirms our theoretical prediction: for one-dimensional projections, the error bound for coarse scrambling only matches that of usual scrambling when the number of points is chosen appropriately relative to the block sizes. Since many dimensions have a block size of $e_j=7$, choosing $m$ as a multiple of 7 synchronizes the error reduction across these dimensions, leading to the observed drops.

\subsubsection{Weighted Smooth Function}
Next, we consider the 100-dimensional weighted function:
\[
g(\bsx) = \prod_{j=1}^{100} \left(1+\frac{1}{j^2}(x_j e^{x_j}-1)\right).
\]
This integrand is dominated by its first few dimensions.
For $d=100$, roughly half of the dimensions in the Sobol' sequence correspond to a block size of $e_j=9$.

The results in Figure~\ref{fig:RQMC_weighted} show that both scrambling methods perform comparably, with convergence rates faster than $O(n^{-1})$ but slightly slower than the ideal $O(n^{-1.5})$. In this scenario, the potential drawbacks of coarse scrambling are negligible. The function's variance is concentrated in the low-index dimensions, which have small block sizes ($e_j=1, 2, \dots$). The performance degradation from the larger block sizes associated with high-index dimensions is insignificant because these dimensions contribute very little to the total variance. This suggests that for functions with low truncation dimension,
coarse scrambling provides a comparable level of accuracy to usual scrambling.

\subsubsection{A function unfavorable to usual scrambling}
First, define the (leftmost) level-$\ell$ Haar wavelet by
\[
h_\ell(x) =
\begin{cases}
2^{\ell/2}, & 0 \le x < 2^{-\ell-1},\\
-2^{\ell/2}, & 2^{-\ell-1} \le x < 2^{-\ell},\\
0, & \text{otherwise.}
\end{cases}
\]
We then define the $19$-dimensional function $h$, our third test integrand, by
\[
h(\bsx)
= h_1(x_{13})h_1(x_{14})h_2(x_{15})h_0(x_{16})h_0(x_{17})h_0(x_{18})h_1(x_{19}).
\]
Thus, the first $12$ dimensions do not affect the function value, and $h$ depends
only on $7$ dimensions, i.e., from $x_{13}$ to $x_{19}$,
whose corresponding block size is $5$ or $6$.
This test function is constructed as a normalized linear combination of Walsh functions so that all of its nonzero Walsh coefficients lie in a single $(u,\bsk)$ block, for which the gain coefficient under the standard scrambling of the Sobol' nets constructed by Joe and Kuo~\cite{JK08} is very large at the sample size considered.
More precisely, we take $u=\{13,\ldots,19\}$ and $\bsk=(1,1,2,0,0,0,1)$.
For $n=2^{16}$, the corresponding gain coefficient equals $2048$ (see \cite{PO21} for the computation).
Since the function is normalized so that $\Var h=1$, the crude Monte Carlo variance is $1/n$.
Hence, at $n=2^{16}$ the predicted RMSE is
\[
\sqrt{2048/2^{16}}=2^{-5/2}\approx 0.1768.
\]

The behavior shown in Figure~\ref{fig:RQMC_anti} is clear.
For the standard scrambled Sobol' nets, the RMSE does not decrease as $n$ increases from $2^{5}$ to $2^{16}$, and the observed values agree with the above prediction.
In contrast, for the coarsely scrambled Sobol' nets, the RMSE decays at essentially the same rate as in plain Monte Carlo, namely $n^{-1/2}$, and its magnitude is comparable (i.e., within an approximately constant factor).
For $n=2^{15}$ and $2^{16}$, over $R=200$ independent trials we always obtained the exact integral value (i.e., the estimator returned $0$), so these points are omitted from the plot.
In an additional experiment with $R=2000$ trials, we observed exactly one nonzero estimate in each case. The resulting RMSE was approximately $0.00395$ in both cases. Thus, the variance is not exactly zero, but extremely small in these cases, and the zero errors observed over $R=200$ trials should be understood as a finite-sampling effect.
Overall, for this test function the coarse scrambling is far more effective than the usual scrambling.

\appendix
\section{Equivalence of ANOVA Formulations}\label{sec:anova-equiv}

This appendix serves two purposes. First, we demonstrate that our formulation of the ANOVA decomposition components is equivalent to the standard formulation via the Walsh basis. Second, we prove the identity in \eqref{eq:beta2}.
For simplicity, we consider the case of a common base, i.e., $\bsb = (b,\dots,b)$, as the generalization to the mixed base case is straightforward.

Let $\wal_{\bsh}(\bsx)$ be the $\bsh$-th Walsh function and $\widehat{f}(\bsh)$ be the corresponding Walsh coefficient of $f$, as defined in \cite[Appendix~A]{DP10book}.
The ANOVA components in \cite[Section~13.3.2]{DP10book}, which we denote by $\beta'_{\bsl}(\bsx)$, are defined for $\bsl \in \NN_0^d$ as
\begin{equation}\label{eq:beta-dash}
\beta'_{\bsl}(\bsx) = \sum_{\bsh \in L_{\bsl}} \widehat{f}(\bsh) \wal_{\bsh}(\bsx),
\end{equation}
where the index set $L_{\bsl}$ is given by
\[
L_{\bsl} = \{ \bsh = (h_1,\dots,h_d) \in \NN_0^d \mid b^{l_j-1} \le h_j < b^{l_j} \text{ if } l_j>0, \text{ and } h_j=0 \text{ if } l_j=0 \}.
\]
Our goal is to show that this definition coincides with ours.
Specifically, for an index $\bsl \in \NN_0^d$,
let $u := \{j \in 1{:}d \mid l_j > 0 \}$ and let $\bsk = (l_j-1)_{j \in u}$. We will show that $\beta'_{\bsl}(\bsx) = \beta_{u,\bsk}(\bsx_u)$, where $\beta_{u,\bsk}$ is defined as in \eqref{eq:beta2}.

First, note that for any $j \notin u$, we have $l_j=0$ and thus $h_j=0$ for all $\bsh \in L_{\bsl}$. Since $\wal_0(x_j) = 1$ and $\wal_{\bsh}(\bsx) = \prod_{j=1}^d \wal_{h_j}(x_j)$, the term $\wal_{\bsh}(\bsx)$ does not depend on $x_j$ for $j \notin u$. Consequently, $\beta'_{\bsl}(\bsx)$ is a function of only $\bsx_u$.

The set $L_{\bsl}$ can be expressed using the inclusion-exclusion principle. Let
\[
\tilde{L}_{\bsl,v}
:= \{ \bsh \in \NN_0^d \mid h_j < b^{l_j-\chi(j \in u \setminus v)} \text{ for } 1 \le j \le d\}
\]
for any $v \subseteq u$.
The indicator function for $\bsh \in L_{\bsl}$ can then be written as
\[
\chi(\bsh \in L_{\bsl}) = \sum_{v \subseteq u} (-1)^{|u|-|v|} \chi(\bsh \in \tilde{L}_{\bsl,v}).
\]
Substituting this into \eqref{eq:beta-dash} yields
\begin{equation}\label{eq:betadash-IE}
\beta'_{\bsl}(\bsx)
= \sum_{v \subseteq u}(-1)^{|u|-|v|} \left( \sum_{\bsh \in \tilde{L}_{\bsl,v}} \widehat{f}(\bsh) \wal_{\bsh}(\bsx) \right).
\end{equation}
The inner sum over $\bsh \in \tilde{L}_{\bsl,v}$ is a partial Walsh series of $f$. As shown in the proof of \cite[Theorem A.11]{DP10book}, this sum has an integral representation:
\begin{align}
\sum_{\bsh \in \tilde{L}_{\bsl,v}} \widehat{f}(\bsh) \wal_{\bsh}(\bsx)
&= \int_{[0,1)^d} f(\bsy) \prod_{j=1}^d b^{l_j-\chi(j \in u \setminus v)}\chi\big(x_j \ominus y_j \in [0,b^{-(l_j-\chi(j \in u \setminus v))})\big) \,d\bsy \notag \\
&= \tilde{\beta}_{u,v,\bsk}(\bsx_u), \label{eq:beta-betadash}
\end{align}
where $\ominus$ is the $b$-adic bitwise subtraction. The integral expression in the first line corresponds precisely to the definition of $\tilde{\beta}_{u,v,\bsk}(\bsx_u)$ in our framework (see \eqref{eq:beta_tilde}).

By substituting \eqref{eq:beta-betadash} into \eqref{eq:betadash-IE}, we obtain
\[
\beta'_{\bsl}(\bsx) = \sum_{v \subseteq u}(-1)^{|u|-|v|} \tilde{\beta}_{u,v,\bsk}(\bsx_u).
\]
The right-hand side is, by definition in \eqref{eq:beta},
equal to $\beta_{u,\bsk}(\bsx_u)$.
This establishes the equivalence $\beta'_{\bsl}(\bsx) = \beta_{u,\bsk}(\bsx_u)$ and completes the proof.

Finally, we prove the identity in \eqref{eq:beta2}. For any $u \subseteq 1{:}d$ and $\bsk \in \NN_0^{|u|}$, we consider the sum of our ANOVA components $\beta_{u,\bsk'}$ over all $\bsk' \le \bsk$ (component-wise). Using the established equivalence $\beta_{u,\bsk'}(\bsx_u) = \beta'_{\bsl(u,\bsk')}(\bsx)$, where $\bsl(u,\bsk')$ is defined such that $l_j = k'_j+1$ for $j \in u$ and $l_j=0$ otherwise, we have
\begin{align*}
\sum_{\substack{\bsk' \in \NN_0^{|u|} \\ \bsk' \le \bsk}} \beta_{u,\bsk'}(\bsx_u)
&= \sum_{\substack{\bsk' \in \NN_0^{|u|} \\ \bsk' \le \bsk}} \beta'_{\bsl(u,\bsk')}(\bsx) \\
&= \sum_{\substack{\bsk' \in \NN_0^{|u|} \\ \bsk' \le \bsk}} \sum_{\bsh \in L_{\bsl(u,\bsk')}} \widehat{f}(\bsh) \wal_{\bsh}(\bsx).
\end{align*}
The union of the disjoint sets $L_{\bsl(u,\bsk')}$ over all $\bsk' \le \bsk$ covers exactly the set of all Walsh indices $\bsh$ where $h_j < b^{k_j+1}$ for all $j \in u$ and $h_j=0$ for $j \notin u$. This corresponds precisely to the index set we previously denoted as $\tilde{L}_{\bsl(u,\bsk),u}$. Therefore, we can combine the summations:
\begin{align*}
\sum_{\substack{\bsk' \in \NN_0^{|u|} \\ \bsk' \le \bsk}} \sum_{\bsh \in L_{\bsl(u,\bsk')}} \widehat{f}(\bsh) \wal_{\bsh}(\bsx)
= \sum_{\bsh \in \tilde{L}_{\bsl(u,\bsk),u}} \widehat{f}(\bsh) \wal_{\bsh}(\bsx) 
= \tilde{\beta}_{u,u,\bsk}(\bsx_u).
\end{align*}
The final equality follows directly from \eqref{eq:beta-betadash}. This confirms \eqref{eq:beta2}.

\section*{Acknowledgments}
The authors acknowledge the use of Google's Gemini 2.5 for editing and language polishing of the manuscript.

\bibliographystyle{siamplain}
\bibliography{ref.bib}

\end{document}